\documentclass[a4paper, 11pt]{article}
\usepackage{stilNaN}

\title{Exotic symplectomorphisms of strong fillings}
\author{Du\v san Drobnjak\footnote{Faculty of Mathematics, University of Belgrade, Studentski trg 16, 11158 Belgrade, Serbia;\\ email: dusan.drobnjak@matf.bg.ac.rs}}

\begin{document}

\maketitle{}

\begin{abstract}
    We prove a criterion for detecting exotic symplectomorphisms of strong fillings of closed contact manifolds. The criterion concerns symplectomorphisms induced by free contact circle actions on the boundary, and it is expressed purely in topological terms. As an application, we show that one-point blow-ups of an infinite family of Brieskorn varieties admit exotic compactly supported symplectomorphisms.
\end{abstract}


\section{Introduction} 

Let $(W,\omega)$ be a symplectic manifold, and let $\Sympc(W,\om)$
denote the group of compactly supported symplectomorphisms of $W$.\ If $W$ has non-empty boundary, compactly supported symplectomorphisms are assumed to be equal to the identity near the boundary. The group $\smcg{W, \om}$ of path components of $\Sympc(W,\om)$ (a path connecting two elements of
$\Sympc(W,\omega)$ is a compactly supported symplectic isotopy) is called the symplectic mapping class group. A symplectomorphism representing a non-trivial mapping class is, in this paper, referred to as exotic.

In his famous paper \cite{gromov1985pseudo}, Gromov established, among many foundational results in symplectic topology, both the existence and non-existence of exotic symplectomorphisms on specific manifolds. For
instance, he showed that the symplectic mapping class groups of the
standard symplectic $4$-disc and of $\mathbb{C}P^2$ are trivial, while
that of the monotone $S^2\times S^2$ is isomorphic to $\mathbb{Z}_2$. Since then, non-trivial symplectic mapping classes have been detected in numerous symplectic manifolds, see e.g. \cite{evans2011stein, chiang2014open, seidel2014exotic, smirnov2023blowups}.

This work builds on the result of \cite{drobnjak2022exotic}, which dealt with exact (Liouville) fillings of closed contact manifolds. The aim of the present paper is to adapt the methods from \cite{drobnjak2022exotic} to strong (not necessarily exact) fillings of closed contact manifolds.

Given a loop of contactomorphisms $\{\fii_t\}_{t\in S^1}$ on the contact boundary, based at the identity, we construct an element of the symplectic mapping class group of the completion of $W$ by taking the time-one map of a suitable Hamiltonian isotopy near the boundary of $W$ and in the cylindrical end of the completion. We denote the symplectic mapping class of the resulting time-one map by $\Theta([\fii_t])$. This construction is due to Biran and Giroux \cite{biran2005symplectic} (see also \cite{uljarevic2017floer}), who considered it in the setting of Liouville domains. It extends directly to strong fillings (see Section~\ref{section:setup}). The main theorem of this paper (Theorem~\ref{theorem:mainIntroduction} below) says that if the loop $\fii_t$ is a free contact circle action that does not satisfy a certain topological condition, then the constructed symplectic mapping class is non-trivial. 

The mentioned topological condition, called topological symmetry, was introduced in \cite{drobnjak2022exotic} by Uljarevi\'c and the author (a special case had already appeared in \cite{uljarevic2017floer}). See Definition~\ref{definition:topSymmetry} for the version used in this paper. It is a property of a contact circle action expressed in terms of the relative homology of the symplectic manifold and a part of its contact boundary.

The precise formulation of the main result is given in the following theorem.

\begin{restatable}{theorem}{mainTheorem}\label{theorem:mainIntroduction}
Let $(W, \om, \lambda)$ be a compact weakly$^+$ monotone symplectic manifold satisfying $H^1(W) = 0$, which is a strong filling of a closed connected contact manifold $(\Sigma, \alpha)$.\ Suppose that $\Sigma$ admits a free contact circle action $\fii_t$ which is not topologically symmetric. Then the class $\Theta([\fii_t])$ is a non-trivial element of $\smcg{\widehat W, \om}$. In particular, if the contact Hamiltonian of $\fii_t$ is strict, then $\Theta([\fii_t])$ is represented by an exotic symplectomorphism of $W$.
\end{restatable}

A related approach to positive contact loops and symplectic mapping classes is developed in \cite{cant2023extensible} for symplectically atoroidal fillings.

As an application of this theorem, we study one-point blow-ups of Brieskorn varieties. We show that, for an infinite family of Brieskorn varieties satisfying an arithmetic condition on the weights, the one-point blow-up admits an exotic compactly supported symplectomorphism (Theorem~\ref{theorem:example}). It demonstrates that the assumptions of Theorem~\ref{theorem:mainIntroduction} are satisfied in natural geometric settings.

The structure of the paper is as follows.\ In Sections~\ref{section:setup} and \ref{section:FloerTheory} we recall the geometric and Floer-theoretic framework used in the paper. In Section~\ref{section:mainTheorem} we prove the main theorem, as well as auxiliary lemmas. Finally, in Section~\ref{section:examples} we apply the main theorem to one-point blow-ups of Brieskorn varieties.

\medskip
\noindent\textbf{Acknowledgements.} I would like to thank Igor Uljarevi\'c for introducing me to this problem and for his guidance and continuous support throughout this work, which will be part of my PhD thesis under his supervision.\ I am grateful to Aleksandra Marinkovi\'c for suggesting blow-ups as a method for constructing examples. 

The author is partially supported by the Ministry of Science, Technological Development and Innovation, Republic of Serbia, through the project 451-03-33/2026-03/200104. 

\textbf{AI statement.} I used LLMs for proofreading, language editing and checking the references. I verified all of their suggestions and take full responsibility for the content of this paper.


\section{Geometric setup} \label{section:setup}

In this section, we recall the basic notions from symplectic and contact geometry that form the geometric setting of our work.\ We also fix the notation that is used throughout the paper. We first recall the definition of a strong filling.

\begin{definition}
    Let $(W^{2n}, \om)$ be a compact symplectic manifold with boundary $\partial W=\Sigma^{2n-1}$. Assume that there exists a 1-form $\lambda$ defined in a neighborhood $\mathcal U \subset W$ of $\Sigma$, such that $\om = \dif\lambda$ in $\mathcal U$ and such that $\alpha = \left.\lambda\right\vert_{\Sigma}$ is a contact form on $\Sigma$, with $\alpha \wedge (\dif\alpha)^{n-1} > 0$ with respect to the boundary orientation of $\Sigma$. We refer to the triple $(W, \om, \lambda)$ as a \emph{strong filling} of the contact manifold $(\Sigma, \xi = \ker\alpha)$.
\end{definition}

On the neighborhood $\mathcal U$, the Liouville vector field $Z$ is defined as the unique vector field that satisfies $i_{Z} \dif\lambda = \lambda$. The flow of $Z$ allows us to identify a neighborhood of the boundary (some subset of $\mathcal U$) with the collar $\Sigma \times (\eta, 1]$, for some $\eta \in (0, 1)$. We equip $\Sigma \times (\eta, 1]$ with coordinates $(x, r)$. Fix the embedding
\[
\iota \colon \Sigma \times (\eta, 1] \hookrightarrow W
\]
such that $\iota^*\lambda = r\alpha$. The completion $\big(\widehat{W}, \widehat{\omega}\big)$ of the strong filling $(W, \om, \lambda)$ is defined by 
\[
\widehat{W} := \left( \Sigma \times (\eta, +\infty) \right) \cup_{\iota} W,
\quad \text{with} \quad
\widehat{\omega} = 
\begin{cases}
\omega, & \text{in } W, \\
d(r\alpha), & \text{in } \Sigma \times (\eta, +\infty).
\end{cases}
\]
We identify the domain of the embedding $\iota$ with its image. In addition, with a slight abuse of notation, we write $\om$ for $\widehat{\om}$, and $\lambda$ for the 1-form on $\Sigma\times(\eta,+\infty)$ given by $r\alpha$. Informally, we refer to the set $\Sigma \times (r_0, +\infty) \subset \widehat{W}$, for large enough $r_0\geqslant \eta$, as the cylindrical end.

We require $W$ to be weakly$^+$ monotone, a condition that guarantees that Floer homology and the naturality isomorphisms (see Section~\ref{section:FloerTheory}) are well defined. First, we recall the definition of the minimal Chern number $N_W$. Denote by $c_1 = c_1(TW)$ the first Chern class of the tangent bundle $TW$ (with respect to any $\om$-compatible almost complex structure). 
The group $c_1(\pi_2(W))$ is a subgroup of $\bb{Z}$, so $c_1(\pi_2(W)) = N_W\bb{Z}$ for a unique integer $N_W\geqslant 0$, called the minimal Chern number of $W$.

\begin{definition} \label{definition:weakMono}
A symplectic manifold $(W^{2n}, \om)$ is called \emph{weakly$^+$ monotone} if one of the following conditions is satisfied:
\begin{enumerate}
    \item[(1)] For some $\tau\geqslant 0$, $(\om-\tau c_1)(\pi_2(W)) = 0$ (i.e., $(W, \om)$ is monotone).
    \item[(2)] The minimal Chern number $N_W$ satisfies either $N_W=0$ or $N_W\geqslant n-1$.
\end{enumerate}
\end{definition}

If $n-1$ is replaced by $n-2$ in part (2) of the definition above, one obtains the weak monotonicity used in \cite{hofer1995floer}, which is called semipositivity in \cite{mcduff2012j}. The continuation maps and the naturality isomorphisms in Section~\ref{section:FloerTheory} require the stronger condition above. This is condition (W$^+$) in, e.g., \cite{seidel1997pi1} and \cite{ritter2016circle}. Other conditions can be used to ensure that Floer theory is well defined (see, for example, the atoroidal setup in \cite{cant2023extensible}).

Now we recall the basics about contact dynamics on $(\Sigma, \xi)$, and its relation to the Hamiltonian dynamics on $(\widehat W, \om)$. A contactomorphism $\phi\colon \Sigma\to\Sigma$ is a diffeomorphism that preserves the contact structure, i.e.\ $\phi_* \xi = \xi$.\ Given a loop of contactomorphisms $\{\fii_t\}_{t\in S^1}$ based at the identity ($\fii_0=\id$), we define the associated contact vector field $Y_t$ by $Y_t = \big(\frac{\dif}{\dif t}\fii_t \big) \circ \fii_t^{-1}$.\ The contact Hamiltonian $h\colon S^1\times \Sigma \to\bb{R}$ is defined by the relation $h_t(x) = -\alpha(Y_t(x))$. In general, the isotopy $\fii_t$ does not preserve the contact form $\alpha$, but rather satisfies $\fii_t^* \alpha = \kappa^h_t \alpha$, for a positive function $\kappa^h_t\colon \Sigma \to (0, +\infty)$ known as the conformal factor. If $\kappa^h_t\equiv 1$, we say that the contact Hamiltonian $h_t$ is strict. 

In this paper, we are particularly interested in loops of contactomorphisms which are contact circle actions.

\begin{definition}
A \emph{contact circle action} $\fii_t$ is a loop of contactomorphisms satisfying
$$
\fii_t\circ\fii_s=\fii_{t+s},
\quad t,s\in S^1=\mathbb R/\mathbb Z.
$$
\end{definition}

Note that the vector field $Y$ that generates a contact circle action is always autonomous, and so is its contact Hamiltonian $h=-\alpha(Y)$. If the above $S^1$-action is free, every orbit of $Y$ is periodic with minimal period 1.

Now we explain the definition of the map $\Theta$, which appears in Theorem~\ref{theorem:mainIntroduction}. Denote by $\Omega(\Cont(\Sigma), \id)$ the loop space of the contactomorphism group $\Cont(\Sigma)=\Cont(\Sigma, \xi)$, based at the identity map. Start with the loop $\fii_t \in \Omega(\Cont(\Sigma), \id)$, $t\in S^1$, $\fii_0 = \fii_1 = \id$. Let $h_t\colon \Sigma \to \bb{R}$ be the contact Hamiltonian that generates the loop $\fii_t$. Choose a symplectic Hamiltonian $H_t$ such that 
$$H\colon S^1\times \widehat{W} \to \mathbb{R}, \ H_t(x,r) = r\cdot h_t(x),\  \text{for}\ (x,r)\in \Sigma\times [1-\eps, +\infty),$$
for some small $\eps>0$. Such a Hamiltonian always exists: for instance, take $H_t(x,r) = \chi(r)\,r\,h_t(x)$ for $r>\eta$, extended by zero to $\widehat{W}\setminus\left(\Sigma \times (\eta, +\infty)\right)$, where $\chi\colon (\eta, +\infty) \to [0,1]$ is a smooth function such that $\chi(r) = 0$ for $r$ near $\eta$, and $\chi(r)=1$ for $r\geqslant 1-\eps$. Let $\phi_t^H\colon \widehat{W}\to\widehat{W}$ denote the Hamiltonian flow of the Hamiltonian vector field $X_t^H$. We use the convention $\dif H_t = i_{X_t^H}\om$. Set $K = \max\limits_{S^1 \times \Sigma} \kappa_t^h(x) \geqslant 1$. Then, this flow satisfies $$\phi_t^H(x,r) = \left(\fii_t(x), \frac{r}{\kappa^h_t(x)} \right), \quad \text{for}\, (x,r)\in \Sigma \times [(1-\eps)K,+\infty).$$
Define the map \begin{equation} \label{eq:ThetaExtension}
    \Theta\colon \pi_0\Omega(\Cont(\Sigma), \id) = \pi_1\Cont(\Sigma) \to \pi_0\Sympc (\widehat W, \om) \quad \text{by} \quad [\fii_t] \mapsto [\phi_1^H].
\end{equation}
We call $\Theta$ the connecting map (see the remark following Lemma~\ref{lemma:mapTheta}).

\begin{remark} 
If $h_t$ is strict, then $K=1$, so $\phi_1^H$ is compactly supported in the interior of $W$. In this case, $\Theta([\fii_t])$ lies in the image of the map $\smcg{W,\om}\rightarrow \smcg{\widehat{W},\om}$ induced by the inclusion $W\hookrightarrow\widehat{W}$. 
\end{remark}

\begin{lemma} \label{lemma:mapTheta}
The map $\Theta$ is well defined.
\end{lemma}

\begin{proof}
    Note that the time-one map $\phi_1^H$ acts as the identity on $\Sigma \times [(1-\eps)K, +\infty)$, since $\fii_1=\id$ and $\kappa_1^h \equiv 1$, and therefore $\phi_1^H$ is an element of $\Sympc(\widehat W, \om)$. 
    
    First, we prove that the isotopy class of $\phi_1^H$ does not depend on the way we extend the Hamiltonian. Let $H$ and $F$ be two extensions of the same contact Hamiltonian $h$. We can take $t\mapsto \phi_t^H\circ \big(\phi_t^F\big)^{-1}\circ \phi_1^F$, $t\in [0,1]$, to be a compactly supported symplectic isotopy between the time-one maps $\phi_1^H$ and $\phi_1^F$.

    Now we prove that elements of the same class in $\pi_1\Cont(\Sigma)$ map to the same class in $\pi_0\Sympc(\widehat W, \om)$. Let $\chi$ be as above. Let $\fii_t^1$ and $\fii_t^2$ be loops of contactomorphisms in the same class in $\pi_1\Cont(\Sigma)$. Let $\fii_t^s$, for $s\in [1,2]$, be a smooth homotopy of based loops between them. For fixed $s$, each loop $\fii_t^s$ is generated by the contact Hamiltonian $h_t^s$, which we can extend to the symplectic Hamiltonian $H_t^s(x,r)$ equal to $\chi(r)rh_t^s(x)$ for $r>\eta$ (and zero elsewhere). The path $\phi_1^{H^s}$ of time-one maps of $H_t^s$ connects the maps $\phi_1^{H^1}$ and $\phi_1^{H^2}$ smoothly through compactly supported symplectomorphisms. 
\end{proof}

\begin{remark}
    In the exact case, the map $\Theta$ can be identified with the connecting map of a certain Serre fibration, which we now explain. If $W$ is a Liouville domain with a Liouville form $\lambda$ and a completion $\widehat W$, denote by $\Symp(\widehat W,\lambda)$ the group of all exact symplectomorphisms of $\widehat{W}$ (symplectomorphisms $\phi$ of $\widehat{W}$ such that $\phi^*\lambda - \lambda$ is exact and compactly supported). The result by Biran and Giroux \cite{biran2005symplectic} (see also \cite{uljarevic2017floer}) says that $\Symp(\widehat W,\lambda)$ is the total space of the Serre fibration with base space $\Cont(\Sigma)$ and fiber over the identity homotopy equivalent to $\Sympc(\widehat W, \om)$. Then the map $\Theta$ from Lemma~\ref{lemma:mapTheta} can be seen as the connecting map $\Delta$ between $\pi_1$ and $\pi_0$ in the long exact sequence of homotopy groups below
    $$
    \ldots \rightarrow \pi_1\Symp(\widehat W,\lambda) \rightarrow \pi_1\Cont(\Sigma) \xrightarrow{\Delta} \pi_0\Sympc(\widehat W, \om) \rightarrow \ldots
    $$
\end{remark}
\vspace{-3mm}

We conclude this section by introducing the notion of topological symmetry. Earlier versions of this concept appeared in \cite{uljarevic2017floer} and \cite{drobnjak2022exotic}.

\begin{definition} \label{definition:topSymmetry}
    Let $\fii_t$ be a free contact circle action on a contact manifold $\Sigma$ with a strong symplectic filling $W$. Let $h$ be the contact Hamiltonian that generates $\fii_t$, and let $N_W$ be the minimal Chern number of $W$. Denote the subset $\Sigma^+$ of the boundary $\Sigma$ as $\Sigma^+=\{x\in\Sigma \mid h(x) > 0\}$. For $k \in \bb{Z}_{2N_W}$ (where $\bb Z_0 = \bb Z$), with a slight abuse of notation, we denote 
    $$\dim H_{k}(W, \Sigma^+; \bb{Z}_2) = \sum
    \dim H_l(W, \Sigma^+; \bb{Z}_2),$$ where the sum goes over all integers $l$ such that $l-k$ is an integer multiple of $2N_W$ (for $N_W>0$, we can interpret it as the cumulative dimension of homology, wrapped modulo $2N_W$).
    
    We say that the action $\fii_t$ is topologically symmetric (with respect to $W$) if there exists $m\in\bb{Z}_{2N_W}$ such that for all $k\in\bb{Z}_{2N_W}$ we have 
    \begin{equation*} 
        \dim H_{k}(W, \Sigma^+; \bb{Z}_2) = \dim H_{m-k}(W, \Sigma^+; \bb{Z}_2).
    \end{equation*}
\end{definition}

Note that the set $\Sigma^+$ defined above does not depend on the choice of the contact form defining the given contact distribution.


\section{Floer theory} \label{section:FloerTheory}

In this section we recall the definition of Floer homology for contact Hamiltonians, which is the main tool that is used in this paper. The well-definedness of this version of Floer theory is guaranteed by the maximum principle developed in \cite{merry2019maximum}. The case of strict contact Hamiltonians, which is required in this paper, was previously considered in \cite{ritter2016circle}.

Let $h_t\colon \Sigma \to \bb{R}$ be a (time-dependent) contact Hamiltonian, parametrized by $t\in S^1$.\ We describe the construction of the Floer homology $HF_*(h)$ associated with a contact Hamiltonian $h_t$.\ This is, in fact, the Hamiltonian Floer homology $HF_*(H)$ associated with a symplectic Hamiltonian $H_t\colon \widehat{W} \to\bb{R}$ which is equal to 
\begin{equation} \label{eq:Hamiltonian}
H_t(x,r) = r\cdot h_t(x), \text{ for } (x, r)\in \Sigma \times [1, +\infty).
\end{equation}
Throughout, $h_t$ is assumed to have no 1-periodic orbits. The function $h_t$ is referred to as the slope of $H_t$, and up to canonical isomorphism, $HF_*(H)$ depends only on the slope \cite[Section~4]{merry2019maximum}.

We now describe the construction of the Floer homology groups for a symplectic Hamiltonian $H_t$ satisfying \eqref{eq:Hamiltonian}. Consider an $S^1$-family of $\omega$-compatible almost complex structures $J_t$, that is, $g_t(p,q) = \om(p,J_t q)$ defines an $S^1$-family of Riemannian metrics on $\widehat{W}$. In addition, $J_t$ is required to be compatible with $\lambda$ on $\Sigma \times[1,+\infty)$ in the sense that $\dif r (J_t p) = -\lambda(p)$ on $\Sigma \times [1,+\infty)$. The usual non-degeneracy condition is assumed: for every fixed point $x$ of the time-one map $\phi_1^H$, $\det(\dif \phi_1^H (x) - \id) \neq 0$.\ Finally, the linearized operator associated with the Floer equation 
\begin{equation} \label{eq:FloerEquation}
    u\colon \bb{R}\times S^1 \to \widehat{W}, \:\: \partial_s u + J_t(u)(\partial_t u - X_t^H(u)) = 0
\end{equation}
is assumed to be surjective (this holds for a generic choice of $J_t$, see \cite{hofer1995floer}).

Denote by $\mathcal L$ the set of all contractible loops in $\widehat{W}$ and by $\mathcal P (H)$ the set of all contractible 1-periodic solutions of the equation $\frac{\dif}{\dif t} x(t) = X_t^H(x(t))$. Consider all the pairs $(x,u)$ such that $x\in \mathcal L$ and $u\colon D^2 \to \widehat{W}$ is a smooth disc (called a capping disc) that bounds the loop $x$, $x(t)=u(e^{2\pi i t})$. On the set of all those pairs we impose the relation $(x,u_1) \sim (x,u_2)$ if the integrals of $\om$ and $c_1$ over the sphere $u_1\#\overline{u_2}\colon S^2\to \widehat{W}$ vanish.
Denote the resulting quotient space by $\widetilde{\mathcal L}$ and its elements by $[x,u]$. Let $\widetilde{\mathcal P}(H)$ be the set of those elements $[x,u]\in\widetilde{\mathcal L}$ such that $x\in \mathcal P (H)$. The symplectic action functional is defined on the set $\widetilde{\mathcal L}$ by
$$\mathcal A_H([x,u]) = - \int\limits_{D^2} u^* \om - \int\limits_0^1 H_t(x(t)) \, \dif t$$ and its critical points are precisely those in $\widetilde{\mathcal P}(H)$. To each element $[x,u]\in \widetilde{\mathcal P}(H)$, we can assign a Conley--Zehnder index $\mu_{\mathrm{CZ}}([x,u])$ by calculating the index of the path of symplectic matrices along the orbit $x$ obtained by linearizing the Hamiltonian equation over a symplectic trivialization of $u^*T\widehat W$. We grade elements of $\widetilde{\mathcal P}(H)$ by the negative Conley--Zehnder index, $\deg([x,u]) = -\mu_{\mathrm{CZ}}([x,u])$.

The $k$-th Floer chain group $CF_k(H)$ is the set of formal sums 
$$\zeta = \sum\limits_{[x,u] \in \widetilde{\mathcal P}(H)} \zeta([x,u]) \langle [x,u] \rangle,\ \zeta([x,u]) \in \mathbb{Z}_2, \ \deg([x,u]) = k$$
such that for all $c\in\mathbb{R}$ the set $\{[x,u] \in \widetilde{\mathcal P}(H) \mid \zeta([x,u])\neq 0, \, \mathcal A_H([x,u]) \geqslant c\}$ is finite. The boundary operator $\partial_k\colon CF_k(H) \to CF_{k-1}(H)$ is given by $$\partial_k\langle[x,u]\rangle = \sum\limits_{\deg([x,u]) - \deg([y,v]) = 1} n([x,u], [y,v]) \langle [y,v] \rangle,$$ where $n([x,u], [y,v])$ represents the number, calculated modulo 2, of Floer trajectories (solutions of \eqref{eq:FloerEquation}) from $[x,u]$ to $[y,v]$, up to $\bb{R}$-translation. The maximum principle \cite{merry2019maximum} ensures the compactness of moduli spaces, and the weak$^+$ monotonicity of $W$ (see Definition~\ref{definition:weakMono}) then guarantees that $\partial^2 = 0$ \cite[Theorem 5.1]{hofer1995floer}. 

Capping discs can be disregarded if we take the grading modulo $2N_W$. The Floer chain complex can be identified with $$CF_k(H) = \bigoplus\limits \Lambda\langle x\rangle,$$ where the sum goes over all orbits $x \in \mathcal P(H)$ such that $\deg(x)-k$ is an integer multiple of $2N_W$. This makes $CF_k(H)$ a finite-dimensional vector space over $\Lambda$. Floer homology, written as $HF(H)$, is defined as the homology of the complex $(CF(H), \partial)$. Here, $\Lambda$ is the Novikov field, defined in the following way. 

Take the homomorphisms $I_\om, I_{c_1} \colon \pi_2(W) \to \mathbb{R}$ which integrate $\om$ and $c_1$ over the sphere, respectively. Denote $\Gamma_0 = \frac{\ker I_{c_1}}{\ker I_\om \cap \ker I_{c_1}}.$ Let $T$ be a formal variable, and define $\Lambda$ to be the set of all formal sums $$\nu = \sum\limits_{A\in \Gamma_0} \nu(A) T^A, \ \nu(A)\in\mathbb{Z}_2,$$ such that $\{A\in \Gamma_0 \mid \nu(A)\neq 0, I_\om(A) < c\}$ is a finite set for every $c\in \mathbb{R}$. It is well known that $\Lambda$ is a field of characteristic 2.

\begin{remark}
    Notationwise, for any $\bb{Z}$-graded homology (Morse or singular) and any $k\in \bb{Z}_{2N_W}$ (where $\bb{Z}_0=\bb{Z}$), we denote $H_k(-) = \bigoplus H_l(-)$ and $\dim H_k(-) = \sum \dim H_l(-)$, where the sums go over all integers $l$ such that $l-k$ is an integer multiple of $2N_W$.
\end{remark}

We now describe the naturality isomorphism, which is an isomorphism of the Floer homology groups associated with a symplectic Hamiltonian $F_t\colon \widehat{W}\to \bb{R}$ generating a loop of Hamiltonian diffeomorphisms $\psi^F$ whose orbits are contractible. Let $H_t$ be a Hamiltonian function on $\widehat{W}$, that will serve as the argument of the isomorphism map, and let $\overline{F}\# H$ be the Hamiltonian that generates the isotopy $\big( \psi_t^F\big)^{-1}\circ \phi_t^H$. In addition to $F_t$, the construction requires a bijection $b\colon \widetilde{\mathcal P}(H) \rightarrow \widetilde{\mathcal P}(\overline{F}\# H)$ such that it is the restriction of a lift $\widetilde{\mathcal L} \rightarrow \widetilde{\mathcal L}$ of the map ${\mathcal L} \rightarrow {\mathcal L}$ given by
$$ x \mapsto \Big(t \mapsto \big(\psi^F_t\big)^{-1} (x(t))\Big).$$

Since the orbits of $\psi^F$ are contractibile, the map above preserves $\mathcal L$, so that it lifts to $\widetilde{\mathcal L}$ (see the proof of \cite[Lemma~2.4]{seidel1997pi1}). This assumption is automatic in the closed case, where the orbits of a Hamiltonian loop are contractible as a consequence of the Arnold conjecture \cite[Lemma~2.2]{seidel1997pi1}, and in the exact \cite[Section~2.7]{uljarevic2017floer} setting it can be dropped, since there one can work with all free homotopy classes of orbits at once.

We define the naturality isomorphism $\mathcal N (F)\colon CF(H)\to CF(\overline{F}\# H)$ that is on generators $[x,v]$ given by $[x,v] \mapsto b([x,v])$. The induced map on the homology will be denoted by the same notation. These maps do not preserve the grading in general, but they do up to a uniform shift $c\in \mathbb{Z}_{2N_W}$ (see \cite{seidel1997pi1} for more details).

Denote the slopes of the Hamiltonians $H_t$ and $F_t$ by $h_t$ and $f_t$, respectively, and assume that $f_t$ is strict (this is the only case needed in this paper). The slope of the Hamiltonian $\overline{F}\# H$ is given by $(h_t-f_t)\circ \fii_t^f,$ where $\fii_t^f$ is the contact isotopy generated by $f_t$. Now, the naturality isomorphism in terms of the contact Hamiltonians has the following form
$$\mathcal N(F) \colon HF_{*}(h) \to HF_{*+c}\big((h-f)\circ \fii^f\big), $$ where $c\in \mathbb{Z}_{2N_W}$ is the uniform shift in grading.


\section{Proof of the main theorem} \label{section:mainTheorem}

In this section we prove the main theorem (stated in the Introduction as Theorem~\ref{theorem:mainIntroduction}).\ Preceding that proof, we formulate and prove a series of lemmas.

\begin{lemma} \label{lemma:zigzag}
    For $s\in[0,1]$, let $h^s_t$ be a smooth path of contact Hamiltonians on $\Sigma$ such that no $h^s$ has 1-periodic orbits. Then $HF_*(h^0) \cong HF_*(h^1)$. 
\end{lemma}

\begin{proof}
    For exact fillings, this is \cite[Theorem~1.1]{uljarevic2022hamiltonian}. In the proof of that theorem, the Floer data that changes is in the cylindrical end, outside a compact region that contains all Floer solutions. Hence the proof can be used for strong fillings on which Floer homology is well defined, in particular to weakly$^+$ monotone ones (see \cite[Remark~1.6]{uljarevic2022hamiltonian}).
\end{proof}

In the following lemma, we formulate a standard result about the isomorphism between Floer and Morse homology. For a Morse function $H$ on $\widehat W$, we denote by $HM_*(H;\Lambda)$ the Morse homology with $\Lambda$ coefficients defined by the flow lines of the gradient vector field $\nabla H=JX^H$. In this convention, a critical point $p$ of $H$ has degree $2n-\text{ind}_H(p)$, where $\text{ind}_H(p)$ is the Morse index of $p$.

\begin{lemma} \label{lemma:FloerMorse}
    Denote by $\Lambda$ the Novikov field over $\bb{Z}_2$. Let $W^{2n}$ be a weakly$^+$ monotone strong filling of a contact manifold $\Sigma$, and let $H\colon \widehat{W} \to \bb{R}$ be a Morse function such that $H(x,r) = rh(x)$, for $(x,r)\in \Sigma\times [1, +\infty)$, where $h\colon \Sigma \to \bb{R}$ is a contact Hamiltonian that has no constant periodic orbits. Then for sufficiently small $\eps > 0$, the groups $HF_{k}(\eps h)$ and $HM_{k+n}(\eps H; \Lambda)$ are isomorphic, for every $k\in\bb{Z}_{2N_W}$.
\end{lemma}

\begin{proof}
    For sufficiently small $\eps$, the contact Hamiltonian $\eps h$ has no 1-periodic orbits, and the Hamiltonian $\eps H$ is $C^2$-small in $W$, which contains all the critical points of $H$. Therefore, all 1-periodic orbits of $\eps H$ are constant. Hence, the Floer homology $HF_*(\eps H) = HF_*(\eps h)$ is canonically isomorphic to the Morse homology. In our conventions, the grading of the Floer complex (given by the negative Conley--Zehnder index) matches the Morse degree $2n-\text{ind}_H$ translated by $-n$, i.e., $\deg = n-\text{ind}_H$. Note that the grading is taken modulo $2N_W$.
    See \cite[Sections 5 and 6]{hofer1995floer} and \cite{ono1995arnold} for details.
\end{proof}

The next lemma is used in the proof of the main theorem, Theorem~\ref{theorem:mainIntroduction}. It allows us to assume that the contact Hamiltonian of the given contact circle action is strict.

\begin{lemma}\label{lemma:strictHamiltonian}
    Let $(W,\om,\lambda)$ be a strong filling of a closed contact manifold $(\Sigma,\alpha)$, and let $g\colon\Sigma\to(\eta,+\infty)$ be a smooth function. Then 
    $$W_1=\widehat W\setminus\big\{(x,r)\in\Sigma\times(\eta,+\infty)\ \big|\ r>g(x)\big\}$$
    is a strong filling of $(\Sigma,g\alpha)$, where $\partial W_1$ is identified with $\Sigma$ through the map $x \mapsto (x, g(x))$. In addition, there is a symplectomorphism $\Phi\colon \widehat W\to \widehat W_1$ with $\Phi\rvert_{W_1} = \id$ such that $\Theta_1 = \Phi_* \circ \Theta$, where $\Theta_1$ is the connecting map associated with $W_1$.
\end{lemma}

\begin{proof}
The Liouville vector field $Z=r\partial_r$ is transverse to $\partial W_1$ and points out of $W_1$. The map $x\mapsto (x, g(x))$ pulls back $\lambda=r\alpha$ to $g\alpha$, so $(W_1, \om)$ is a strong filling of $(\Sigma, g\alpha)$.

Set $\eta_1=\frac{\eta}{\min_\Sigma g}\in(0,1)$. The embedding
\[
    \iota_1\colon\Sigma\times(\eta_1,1]\hookrightarrow W_1,\qquad \iota_1(x,\bar r)=\big(x,\bar rg(x)\big),
\]
satisfies $\iota_1^*\lambda=\bar rg\alpha$, and we use it in the construction of $\widehat{W}_1$. Define $\Phi\colon\widehat{W}\to\widehat{W}_1$ by
\[
    \Phi|_{W_1}=\id,\qquad \Phi(x, r)=\bigg(x, \frac{r}{g(x)}\bigg)\text{ for }(x, r) \in \Sigma\times(\eta, +\infty), \text{ when } r>\eta_1 g(x).
\]
The two formulas agree on the image of $\iota_1$, so $\Phi$ is well defined and smooth. It is a symplectomorphism since $\Phi^*(\bar rg\alpha)=r\alpha$.

Let $[\fii_t]\in\pi_1\Cont(\Sigma,\ker \alpha)$, and let $h_t$ be its contact Hamiltonian with respect to $\alpha$. Its contact Hamiltonian with respect to $g\alpha$ is $gh_t$. Fix a small $\eps_1>0$ and choose $H_t$ such that $H_t(x,r)=rh_t(x)$ for $r\geqslant(1-\eps_1)\min\{1, \min_\Sigma g\}$. Then $H_t\circ\Phi^{-1}(x,\bar r)=\bar rg(x)h_t(x)$ for $\bar r\geqslant 1-\eps_1.$
Hence, $\Theta_1([\fii_t]) = [\phi_1^{H\circ \Phi^{-1}}] = [\Phi\circ\phi_1^{H}\circ\Phi^{-1}] =\Phi_*\Theta([\fii_t])$.
\end{proof}

The next lemma allows us to pass from Morse homology to singular homology via the isomorphism between the corresponding homology groups.

\begin{lemma} \label{lemma:MorseGluing}
     Let $W$ be a strong filling of a contact manifold $(\Sigma, \alpha)$. Let $h\colon \Sigma \to \bb{R}$ be a strict contact Hamiltonian that generates a free contact circle action $\fii_t$. Let $H\colon \widehat{W}\to \bb{R}$ be a Morse function that for $(x,r)\in \Sigma\times [1,+\infty)$ satisfies $H(x,r)=rh(x)$. Then $HM_*(H; \Lambda) \cong H_*(W, \Sigma^+; \Lambda)$, where $\Sigma^+=\{x\in\Sigma\mid h(x)>0\}$.
\end{lemma}

\begin{proof}
Since $h$ is a strict contact Hamiltonian generating a free contact
circle action, $0$ is a regular value of $h$ by
\cite[Lemma~3.5]{drobnjak2022exotic}.

The proof follows that of
\cite[Proposition~4.2, Steps~2--7]{drobnjak2022exotic} with two minor
modifications. First, the coefficient field $\bb{Z}_2$ is replaced by
the Novikov field $\Lambda$. The proof uses only that the coefficient
ring is a field (through dimension computations in Steps~5 and 6). Second, our symplectic completion is constructed using the collar $\Sigma\times(\eta,+\infty)$ instead of $\Sigma\times(0,+\infty)$. The gluing construction in Step~3 is modified only by replacing the
collar $\Sigma\times(0,R]$ with $\Sigma\times(\eta,R]$, and the connecting cylinder $\Sigma\times\bb{R}$ with $\Sigma\times \big( \log\frac{\eta}{R}, -\log\frac{\eta}{R} \big)$ Since all identifications are defined in the same way, no further changes are required, and the arguments in Steps~4--7 remain valid.
\end{proof}

In the next lemma, we see that the notion of topological symmetry (see Definition~\ref{definition:topSymmetry}) with coefficients in a field of characteristic 2 is essentially the same as that with $\bb{Z}_2$ coefficients.

\begin{lemma} \label{lemma:symmetry}
    Denote by $\Lambda$ the Novikov field genereted over $\bb{Z}_2$. 
    Let $X$ be a topological pair with finitely generated integral homology, and let $\Lambda$ be any field of characteristic 2. If $H_n(X; \Lambda) \cong \Lambda^d$, for some $n,d\in\bb{N}_0$, then $H_n(X; \bb{Z}_2) \cong \bb{Z}_2^d$.
\end{lemma}

\begin{proof} 
    Write 
    $$H_n(X; \bb{Z}) \cong \bb{Z}^a \oplus\bigoplus\limits_{k \in \bb{N}} \bb{Z}_{2^k}^{a_k} \oplus T_1 \quad \text{and} \quad H_{n-1}(X; \bb{Z}) \cong \bb{Z}^b\oplus\bigoplus\limits_{k\in \bb{N}} \bb{Z}_{2^k}^{b_k} \oplus T_2,$$
    where by $T_1$ and $T_2$ we denote the torsion groups consisting of direct sums of cyclic groups of the form $\bb{Z}_{c}$, for odd integers $c\in\bb{N}$.
    We apply the universal coefficient theorem (see \cite[Theorem~3A.3]{hatcher2002algebraic}) with $\Lambda$ coefficients.    
    \begin{alignat*}{1}
        \Lambda^d \cong H_n(X;\Lambda) \cong &\left( H_n(X;\bb{Z}) \otimes \Lambda \right) \oplus \Tor\left(H_{n-1}(X;\bb{Z}), \Lambda\right) \\
        \cong &\Big(\Lambda^a \oplus \bigoplus\limits_{k\in \bb{N}} \Lambda^{a_k}\Big) \oplus \bigoplus\limits_{k\in\bb{N}} \Lambda^{b_k},
    \end{alignat*}
    which implies the equality 
    $$d = a+\sum\limits_{k\in\bb{N}} (a_k+b_k).$$
    The universal coefficient theorem applied again with the $\bb{Z}_2$ coefficients gives
    \begin{alignat*}{1}
        H_n(X;\bb{Z}_2) \cong &\left( H_n(X;\bb{Z}) \otimes \bb{Z}_2 \right) \oplus \Tor\left(H_{n-1}(X;\bb{Z}), \bb{Z}_2\right) \\
        \cong &\Big(\bb{Z}_2^a \oplus \bigoplus\limits_{k\in \bb{N}} \bb{Z}_2^{a_k}\Big) \oplus \bigoplus\limits_{k\in\bb{N}} \bb{Z}_2^{b_k} \cong \bb{Z}_2^d.
    \end{alignat*}
\end{proof}
\vspace{-3mm}

\begin{lemma} \label{lemma:FloerMorseSingular}
    Let $\fii_t\colon \Sigma \to \Sigma$ be a free contact circle action on a closed contact manifold $(\Sigma, \alpha)$ and let $h\colon \Sigma\to \bb{R}$ be its contact Hamiltonian. Assume that $h$ is strict with respect to $\alpha$. If $W$ is a weakly$^+$ monotone strong filling of $(\Sigma, \alpha)$, then for sufficiently small $\eps>0$, $$\dim_\Lambda HF_*(\eps h) = \dim H_{*+n}(W, \Sigma^+; \mathbb{Z}_2)  \ \text{and} \ \dim_\Lambda HF_*(-\eps h) = \dim H_{*+n}(W, \Sigma^-; \mathbb{Z}_2),$$
    where $\Sigma^\pm=\{x\in\Sigma \mid \pm h(x) > 0\}$.
\end{lemma}

\begin{proof}
    Let $H\colon \widehat W\to \bb{R}$ be a Morse function such that $H(x,r) = rh(x)$ for $(x,r)\in \Sigma\times[1,+\infty)$, and set $H_\eps = \eps H$. Since the action $\fii_t$ is free, none of the contact Hamiltonians along the path $\eps \mapsto \eps h$, for $\eps \in (0,1)$, admits 1-periodic orbits.\
    It follows from Lemma~\ref{lemma:zigzag} that the Floer homology groups $HF_*(\varepsilon h)$ are all isomorphic for $\varepsilon\in(0,1)$. The same argument applied to the path $\varepsilon\mapsto -\varepsilon h$ shows that the groups $HF_*(-\varepsilon h)$ are likewise all isomorphic for $\varepsilon\in(0,1)$.
    By Lemma~\ref{lemma:FloerMorse}, for sufficiently small $\eps>0$, we have isomorphisms 
    \begin{equation}
        HF_*(\pm \eps h) \cong HM_{*+n}(\pm H_\eps; \Lambda). \label{eq:FloerMorse}
    \end{equation}
    Moreover, Lemma~\ref{lemma:MorseGluing} provides us with isomorphisms
    \begin{equation} \label{eq:MorseSingular}
        HM_*(\pm H_\eps; \Lambda) \cong H_*(W, \Sigma^{\pm}; \Lambda).
    \end{equation}

Finally, the combination of equations $\eqref{eq:FloerMorse}$ and $\eqref{eq:MorseSingular}$, together with the result of Lemma~\ref{lemma:symmetry} gives us equalities
\begin{equation*}
    \begin{split}
        \dim_\Lambda HF_*(\pm \eps h) = \dim_\Lambda HM_{*+n}(\pm H_\eps; \Lambda) &= \dim_\Lambda H_{*+n}(W, \Sigma^{\pm}; \Lambda) \\ &= \dim H_{*+n}(W, \Sigma^{\pm}; \bb{Z}_2).
    \end{split}
\end{equation*}
\end{proof}
\vspace{-3mm}

Now we are ready to prove the main theorem, stated in the Introduction as Theorem~\ref{theorem:mainIntroduction}. We restate it here for the convenience of the reader.

\mainTheorem*

\begin{proof}
    Denote by $\alpha = \left.\lambda\right|_{\Sigma}$ the contact form on $\Sigma$ and let $h$ be the generating contact Hamiltonian of $\fii_t$, with respect to the form $\alpha$. 
    Take the contact form $\beta$ defined pointwise by averaging $\alpha$ over the action $\fii_t$, i.e.
    $$\beta_x = \int\limits_{S^1} \left(\fii_t^* \alpha \right)_x \dif t,$$
    for $x\in\Sigma$. It follows that the form $\beta$ is invariant under the action $\fii_t$. The forms $\alpha$ and $\beta$ generate the same contact distribution, which means that $\beta = f\alpha$ for a positive function $f$. Choose a constant $C>\frac{\eta}{\min_\Sigma f}$. By Lemma~\ref{lemma:strictHamiltonian}, with $g=Cf$, without loss of generality we can replace $\alpha$ by $C\beta$ and $W$ by $W_1$, i.e. we can assume that $h$ is a strict contact Hamiltonian. The manifolds $W$ and $W_1$ are both deformation retracts of $\widehat W$ (on which $c_1$ and $\om$ are defined) and therefore $W_1$ satisfies the same assumptions as $W$. Moreover, the set $\Sigma^+$ does not depend on the contact form, and $(W, \Sigma^+)$ and $(W_1, \Sigma^+)$ are diffeomorphic via the flow of the Liouville vector field. Therefore, $\fii_t$ is not topologically symmetric with respect to $W_1$.
    
    By the discussion in Section~\ref{section:setup}, the class $\Theta([\fii_t])$ is represented by the time-one map $\phi_1^H$ of the flow $\phi_t^H$ associated with a Hamiltonian $H\colon \widehat{W}\to \bb{R}$ that is equal to 
    $$H(x,r) = r\cdot h(x),\ \text{for} \ (x,r)\in \Sigma\times [1-\delta, +\infty),$$
    for some $\delta>0$. We choose $H$ to be a Morse function.
    
    Assume the contrary, that the class $\Theta([\fii_t])$ is trivial in $\pi_0\Sympc(\widehat W, \om)$. 
    Then there exists a compactly supported symplectic isotopy $\Psi$, 
    $$\Psi\colon [0,1] \to \Sympc(\widehat W, \om), \ \Psi_0 = \id, \ \Psi_1 = \phi_1^H.$$
    
    Since the first cohomology of $W$ vanishes and $\Sigma$ is connected, the symplectic isotopy $\Psi_t$ is a Hamiltonian isotopy. Let it be generated by a compactly supported Hamiltonian $G_t\colon \widehat{W}\to \bb{R}$. In particular, $\phi_0^G = \id$ and $\phi_1^G = \Psi_1 = \phi_1^H$. Moreover, $G_t(x,r) = 0$, for $(x,r)\in \Sigma\times [R,+\infty)$ and for some $R\geqslant 1$. Applying Lemma~\ref{lemma:strictHamiltonian} with $g\equiv R$, we may assume that $R=1$.
    
    The Hamiltonian $F\colon \widehat{W}\to \bb{R}$ with the flow $\phi_t^F = \phi_t^H \circ \left(\phi_t^G\right)^{-1}$ generates a loop of Hamiltonian diffeomorphisms and we use it to establish the naturality isomorphism $\mathcal N (F)$ described in Section~\ref{section:FloerTheory}. We have $F\rvert_{\Sigma\times [1,+\infty)} = H\rvert_{\Sigma\times [1,+\infty)}$. In particular, the slope of $F$ is $h$ and $\kappa^f = \kappa^h = 1$. 
    
    Let $\eps \in (0,1)$ be sufficiently small, so that Lemma~\ref{lemma:FloerMorseSingular} applies. The naturality isomorphism $\mathcal N(F)$ applied to the Hamiltonian $\eps H$ gives an isomorphism 
    \begin{equation} \label{eq:naturalityIso}
    HF_*(\eps h) \cong HF_{*+c}((\eps h - h)\circ \fii_t) \cong HF_{*+c}((\eps -1)h),\ \text{for some} \ c\in\bb{Z}_{2N_W},   
    \end{equation}
    since a strict autonomous contact Hamiltonian is invariant under its own flow.
    Moreover, since the contact action is free, the contact Hamiltonian $h$ has no closed orbits of negative period greater than $-1$ and Lemma~\ref{lemma:zigzag} can be applied to connect Hamiltonians with the slopes $(\eps-1)h$ and $-\eps h$. The isomorphism \eqref{eq:naturalityIso} can be rewritten as 
    \begin{equation} \label{eq:HFiso}
        HF_*(\eps h) \cong HF_{*+c}((\eps-1)h) \cong HF_{*+c}(-\eps h).
    \end{equation}
    By Lemma~\ref{lemma:FloerMorseSingular}, the isomorphism \eqref{eq:HFiso} reduces to the following equality between dimensions of singular homology groups $$\dim H_{*+n}(W, \Sigma^+; \bb{Z}_2) = \dim H_{*+c+n}(W, \Sigma^-; \bb{Z}_2).$$
    Using Lefschetz duality \cite[Theorem~3.43]{hatcher2002algebraic} and the universal coefficient theorem \cite[Theorem~3.2]{hatcher2002algebraic}, we obtain the following chain of equalities
    \begin{alignat*}{1}
        &\dim H_{*}(W, \Sigma^+; \bb{Z}_2) = \dim H_{*+c}(W, \Sigma^-; \bb{Z}_2) = \dim H^{2n-(*+c)}(W, \Sigma^+;\bb{Z}_2) \\
        =& \dim \Hom(H_{2n-c-*}(W, \Sigma^+;\bb{Z}_2), \bb{Z}_2) = \dim H_{2n-c-*}(W, \Sigma^+; \bb{Z}_2),
    \end{alignat*}
    making the action $\fii_t$ topologically symmetric (for $m=2n-c$), which is a contradiction. The naturality isomorphism $\mathcal N(F)$ is defined only if the orbits of $\fii_t$ are contractible in $W$ (see Section~\ref{section:FloerTheory}). Otherwise, for small $\eps$, the 1-periodic orbits of $\overline{F}\#(\eps H)$ are the loops $t\mapsto(\phi^F_t)^{-1}(p)$, where $p$ is a critical point of $H$, and none of them is contractible. Hence the Floer complex (which is generated by contactible orbits) is empty and $HF_*((\eps-1) h)\cong 0$. Similarly as above, one gets that $\fii_t$ is topologically symmetric. This proves that $\Theta([\fii_t])$ is non-trivial in $\smcg{\widehat W, \om}$. 
    
    If the contact Hamiltonian of $\fii_t$ with respect to $\alpha$ is strict, then $\Theta([\fii_t])$ is represented by a symplectomorphism that is supported in the interior of $W$ (see Section~\ref{section:setup}). Its class in $\smcg{W, \om}$ is non-trivial, since it is mapped to $\Theta([\fii_t])$ by the map induced by the inclusion $W\hookrightarrow \widehat W$.
\end{proof}

\begin{remark}
Note that the vanishing of the first cohomology is used here to pass from Hamiltonian to symplectic mapping classes. It is enough to assume that the appropriate flux homomorphism is trivial. This condition is satisfied, for example, in the atoroidal setting considered in \cite{cant2023extensible}.
\end{remark}


\section{Application to blow-ups of Brieskorn varieties} \label{section:examples}

In this section we apply Theorem~\ref{theorem:mainIntroduction} to a class of symplectic manifolds that are constructed as blow-ups of Brieskorn varieties. These manifolds provide examples of weakly$^+$ monotone symplectic fillings whose symplectic mapping class groups contain non-trivial elements. Brieskorn manifolds and varieties are well-established objects in contact and symplectic topology (see, for example, \cite{boyer2016brieskorn, kwon2016brieskorn,uebele2016symplectic}).

Let $n$ be a positive integer. For a vector of positive integers $a = (a_0$, $a_1$, \ldots, $a_n)$, define the Brieskorn variety
$$
V = V(a)=\left\{ z = (z_0,\ldots, z_n)\in\mathbb{C}^{n+1}\:   \big| \: z_0^{a_0} + \ldots + z_n^{a_n} =\varepsilon\:, \: \norm{z}\leqslant 1 \right\},
$$ where $\varepsilon>0$ is a sufficiently small number. This is a Liouville domain of dimension $2n$ with contact boundary $\partial V$, contactomorphic to the so-called Brieskorn manifold
$$
\Sigma = \Sigma(a)=\left\{ z = (z_0,\ldots, z_n)\in\mathbb{C}^{n+1}\:\big| \: z_0^{a_0} + \ldots + z_n^{a_n} =0\:, \: \norm{z}= 1 \right\}, 
$$
endowed with the contact form $\alpha_a = \frac{i}{4\pi L} \sum_{j=0}^{n} a_j (z_j\dif \bar{z}_j - \bar{z}_j \dif z_j)$, where $L$ is the least common multiple of $a_0, a_1, \ldots, a_n$. 
Consider the contact circle action $\fii^a_t$ on $\Sigma$ given by
\begin{equation} \label{eq:ReebBrieskorn}
    \fii_t^a(z_0, \ldots, z_n) = \left(e^{2\pi i\frac{L}{a_0} t}\cdot z_0, \ldots, e^{2\pi i\frac{L}{a_n} t}\cdot z_n \right).
\end{equation}
This is the Reeb flow of the contact form $\alpha_a$ on the Brieskorn manifold. It is 1-periodic and strict with respect to the form $\alpha_a$. For a suitable choice of the vector $a$, this action is free, which is described in the following lemma.

\begin{lemma} \label{lemma:freeReeb}
    Let $a=(a_0, a_1, \ldots, a_n)$ be a vector of positive integers and let $L$ be the least common multiple of $a_0, a_1, \ldots, a_n$. Let $\fii_t^a$ be the Reeb flow on the Brieskorn manifold $\Sigma(a)$. Then the $S^1$-action $\fii_t^a$ is free if and only if the numbers $\frac {L}{a_i}$ and $\frac{L}{a_j}$ are relatively prime, for all $0 \leqslant i, j \leqslant n$, $i\neq j$. 
\end{lemma}

\begin{proof}
In the Sasakian literature, a Sasakian structure is called regular
when the Reeb vector field generates a free $S^1$-action. For
Brieskorn manifolds, regularity is equivalent to the condition in the
statement of the lemma, see \cite[Proposition~9.3.22]{boyer2008sasakian}.
\end{proof}

Let $W=W(a)$ be a symplectic blow-up of $V$ at an interior point. Then $W$ is a strong filling of the contact manifold $(\Sigma, \alpha_a)$. The next theorem gives sufficient conditions on the vector $a$ for $W$ to admit an exotic symplectomorphism.

\begin{theorem} \label{theorem:example}
    Let $n\geqslant 3$ be an integer. Let $a=(a_0, a_1, \ldots, a_n)$ be a vector of positive integers, such that the numbers $\frac{L}{a_i}$ and $\frac{L}{a_j}$ are relatively prime for $0\leqslant i, j \leqslant n$, $i\neq j$, where $L$ is the least common multiple of $a_0, a_1, \ldots, a_n$. Denote $\mu = \prod_{j=0}^n (a_j-1)$. Then the manifold $W(a)$ admits an exotic symplectomorphism if any of the following conditions are met:
    \begin{equation} \label{eq:conditions}
        n \text{ is odd and } \mu\geqslant 1 \,\text{ OR }\, n \text{ is even and } \mu\geqslant 2.
    \end{equation}
\end{theorem}

\begin{proof}
    We want to apply Theorem~\ref{theorem:mainIntroduction} to the manifold $W=W(a)$. Consider the Reeb flow on $\Sigma=\Sigma(a)$, given by \eqref{eq:ReebBrieskorn}. This is a contact circle action, and by Lemma~\ref{lemma:freeReeb} it is free. The set $\Sigma^+$ is empty in this case, since the contact Hamiltonian of a Reeb flow is the constant function $-1$. By Milnor's theorem \cite{milnor1968singular}, the manifold $V$ has the homotopy type of a wedge of $\mu$ spheres $S^n$, where $\mu=\prod_{j=0}^n(a_j-1)$. In particular, since $n\geqslant 3$, we have $H^2(V;\mathbb{Z})=0$, and hence $c_1(V)=0$. Moreover, $V$ and $W$ are simply connected, so $H^1(W) = 0$. The minimal Chern number of $W$ is $N_W = n-1$, so $W$ is weakly$^+$ monotone (in fact, it turns out that $W$ is monotone as well).

    We shall now compute the homology of $W$, since the topological symmetry condition depends on the homology $H_*(W, \Sigma^+; \bb{Z}_2) \cong H_*(W; \bb{Z}_2)$. Blowing up at one point adds the reduced homology of the exceptional divisor $E\cong\mathbb{C}P^{n-1}$. Thus it changes the homology groups of $V$ in all even dimensions in the set $\{2,4,6,\ldots, 2n-2\}$, by adding an additional $\bb{Z}_2$ generator (of the exceptional divisor $\bb{C}P^{n-1}$) in each one. Therefore,
    $$
    H_k(W;\bb{Z}_2) \cong H_k(V; \bb{Z}_2) \oplus \widetilde{H}_k(\bb{C}P^{n-1};\bb{Z}_2) \cong 
    \begin{cases}
    \mathbb Z_2, & k\in \{0,2,4,\ldots,2n-2\}\setminus\{n\},\\
    \mathbb Z_2^{\mu+1}, & k=n,\; n\text{ even},\\
    \mathbb Z_2^\mu, & k=n,\; n\text{ odd},\\
    0, & \text{otherwise}.
    \end{cases}
    $$
    The homology wrapped modulo $2N_W = 2n-2$ is 
    $$
    H_k(W;\bb{Z}_2) \cong 
    \begin{cases}
    \mathbb{Z}_2^2, & k=0, \\
    \mathbb Z_2, & k\in \{2,4,\ldots,2n-4\}\setminus\{n\},\\
    0, & k\in \{1,3,\ldots,2n-3\}\setminus\{n\},\\
    \mathbb Z_2^{\mu+1}, & k=n,\; n\text{ even},\\
    \mathbb Z_2^\mu, & k=n,\; n\text{ odd}.
    \end{cases}
    $$

    \begin{figure}[H]
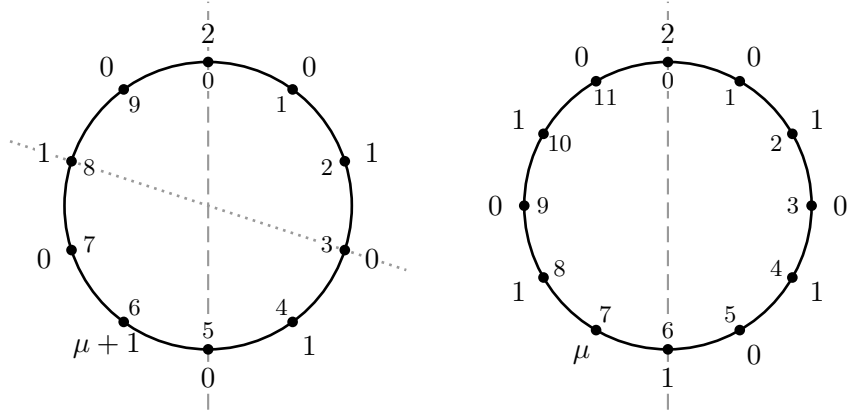
 
    \centering
    
    \twoCircles
    
    \caption{Examples of sequences of dimensions, wrapped modulo $2n-2$, for $n=6$ (left) and $n=7$ (right). In both circles, the dashed line represents the axis of symmetry corresponding to the case $\mu=0$. In the left circle, the dotted line represents the axis of symmetry corresponding to the case $\mu=1$.}
    \label{figure:topologicalSymmetry}
    \end{figure}
    
    \FloatBarrier

    We can now see that if the condition \eqref{eq:conditions} is met, the action $\fii_t^a$ is not topologically symmetric.\ In these cases, since $\fii_t^a$ is strict, Theorem~\ref{theorem:mainIntroduction} implies that $\Theta([\fii_t^a])$ is represented by an exotic symplectomorphism of $W$. The illustration of the wrapped sequence of dimensions is presented for an even and odd $n$ in Figure~\ref{figure:topologicalSymmetry}.
\end{proof}

\begin{remark}
    One can obtain results analogous to Theorem~\ref{theorem:example} when
    $W$ is the symplectic blow-up of $V$ at several points rather than a
    single one. More precisely, for $m\geqslant 1$, let $W=W(a,m)$ denote
    the symplectic blow-up of $V(a)$ at $m$ generic interior points. Then
    the wrapped homology modulo $2N_W=2n-2$ is given by
    $$
    H_k(W;\bb{Z}_2) \cong 
    \begin{cases}
    \mathbb{Z}_2^{m+1}, & k=0, \\
    \mathbb Z_2^m, & k\in \{2,4,\ldots,2n-4\}\setminus\{n\},\\
    0, & k\in \{1,3,\ldots,2n-3\}\setminus\{n\},\\
    \mathbb Z_2^{\mu+m}, & k=n,\; n\text{ even},\\
    \mathbb Z_2^\mu, & k=n,\; n\text{ odd}.
    \end{cases}
    $$ 
    The resulting sequence of dimensions is not symmetric
    whenever condition \eqref{eq:conditions} is satisfied. Here is also $N_W=n-1$, so $W$ is weakly$^+$ monotone (although not necessarily monotone).
\end{remark}

\begin{example}
    The assumption on the weights is satisfied, for example, by $a=(d,d,\ldots,d)$ with $d\geqslant 2$, since then $L=d$ and $L/a_i=1$ for all $i$. Hence Theorem~\ref{theorem:example} applies to the blow-up of the Milnor fiber of $z_0^d+z_1^d+\ldots+z_n^d$ for every $n\geqslant 3$ whenever $d\geqslant 2$ if $n$ is odd, and $d\geqslant 3$ if $n$ is even.
\end{example}

\begin{remark}
None of the Brieskorn manifolds satisfying theconditions of Theorem~\ref{theorem:example} is a topological sphere. Brieskorn's criterion (see, for example, \cite[Proposition~3.6]{kwon2016brieskorn}) characterizes when $\Sigma(a)$ is homeomorphic to a sphere. The condition in this
criterion requires some $a_i$ to be coprime to all the other $a_j$. Under the conditions of Theorem~\ref{theorem:example} this is impossible (such $a_i \geqslant 2$ would divide $L/a_j$).
\end{remark}

\begin{remark}
Note that $\mu=1$ if and only if $a=(2,2,\ldots,2)$. In this case, the Brieskorn variety is the cotangent disc bundle of a standard sphere, $V=D^*S^n$. When $n$ is odd, Theorem~\ref{theorem:example} detects an exotic symplectomorphism on the blow-up of $V$. Without performing the blow-up, i.e., when working on $V$, our method does not yield an exotic symplectomorphism, since $H_*(V;\bb{Z}_2)$ is symmetric. Nevertheless, the symplectic mapping class group of $V$ is known to be non-trivial (in fact, the standard Dehn--Seidel twist along the zero-section has infinite order \cite{uljarevic2019viterbo}).
\end{remark}

\bibliography{references.bib} 
\bibliographystyle{alpha}

\end{document}